\documentclass[12pt]{amsart}
\usepackage{txfonts}
\usepackage{amssymb}
\theoremstyle{plain}

\newtheorem{Thm}{Theorem}

\newtheorem{Def}[Thm]{Definition}

\begin{document}

\title[constants and heat flow]
{constants and heat flow on graphs}

\author{Li Ma}
\address{Li MA,  School of Mathematics and Physics\\
  University of Science and Technology Beijing \\
  30 Xueyuan Road, Haidian District
  Beijing, 100083\\
  P.R. China }
 \email{lma17@ustb.edu.cn}

\dedicatory{This paper is dedicated to our teacher Prof.Hou Zixin on the occasion of his 80th birthday}

\thanks{Li Ma's research is partially supported by the National Natural
  Science Foundation of China (No.11771124)}

\begin{abstract}
In this article, we
first introduce the concepts of vector fields and their divergence, and we recall the concepts of the gradient, Laplacian operator, Cheeger constants, eigenvalues, and heat kernels on a locally finite graph $V$. We give a projective characteristic of the eigenvalues. We also give an extension of Barta Theorem. Then we introduce the mini-max value of a function on a locally finite and locally connected graph. We show that for a coercive function on on a locally finite and locally connected graph, there is a mini-max value of the function provided it has two strict local minima values.
We consider the discrete Morse flow for the heat flow on a finite graph in the
locally finite graph $V$. We show that under suitable assumptions on the graph one has a
weak discrete Morse flow for the heat flow on $S$
 on any time interval. We also study the heat flow with time-variable potential and its discrete Morse flow.
We propose the concepts of harmonic maps from a graph to a Riemannian manifold and pose some open questions.


\end{abstract}

\maketitle
\tableofcontents

\newpage

\section{Introduction}
  Data in mathematics is considered as function from a set to a vector space such as $R^n$, the n-dimensional Euclidean space. The structure of the data is reflected by the spectrum of combination Laplacian on graphs and related evolution equations. The spectrum may be characterized by the topology of the underlying working space defined on the graph. In the recent interesting work \cite{AY}, Allen, Lippner, Chen, Y. et al. consider the population evolution using the analysis of graph theory. We consider here some basic ingredients of analysis on graphs such as the concepts of vector fields, mini-max value of a function defined on a locally finite graph, the evolution of heat flow of datum from graphs, and related concepts of harmonic maps from a graph to a Riamannian manifold.

  This work consists of two parts. One is about the concepts of vector fields and their divergence, eigenvalue estimations \cite{C} \cite{D}, mini-max value, and related problem on graphs, and the other is for heat equation.
Recently we have seen a lot works on heat equation and eigenvalue estimations and related problem on graphs. We pay more attention to the Discrete Morse flow method to the heat equations.  We give a projective characteristic of the eigenvalues. We shall touch some small part of this development and we discuss the Cheeger constants, eigenvalues, and heat kernels on a finite graph in the
locally finite graph $V$. For the estimations about Cheeger constants and isoperimetric inequality, we also refer to the interesting work of Dodziuk \cite{D} and the recent interesting book \cite{G}.
One may also refer to know more about the surjective property of the combination Laplacian operator \cite{D2}.
We show that under suitable assumptions on the initial data $u(0)$ one has a
 discrete Morse flow for the heat flow on $S$
 on any time interval. We also study the heat flow with time-variable potential and its discrete Morse flow. Perhaps the interesting thing here is that we introduce the concepts of vector fields and their divergence on a locally finite graph.

We have studied the discrete Morse flows for 2-dimensional Ricci flows and porous medium equations in \cite{MW}.  As a numerical method, the discrete Morse flow method can be used to study many problems from applied mathematics, for example, problems from the population evolution, control theory and complex systems (see \cite{O}  for more references).
The precise results are stated and proved in the following
sections.  There are relatively few results about computational models for the
time variable heat flow in any graph, one may refer to \cite{BL} (see also \cite{BK} and \cite{OS} for related references). The questions about suitable discrete Morse flows for Yamabe flows and related problems \cite{A} \cite{Ar}  \cite{DK}  are still in investigation and the difficulty underlying this flows is the lack of compactness.

The plan of this paper is as
follows: In section \ref{sect1}, we discuss some elementary part of analysis on finite graphs and heat flow.
In section~\ref{sect6}, we introduce the discrete Morse flow method to the heat flow with prescribed reaction term
with the parabolic
initial data and boundary data. The main result of this section is Theorem \ref{mali4}.

\section{Part 1; vector fields and Eigenvalues}\label{sect1}

\subsection{locally finite graphs as metric spaces}
A graph $G=(V,E)$ is a pair of the vertex-set $V$ and the edge-set
$E$. Each edge is an unordered pair of two vertices. If there is an
edge between $x$ and $y$, we write $x \sim y$. We assume that $G$ is
{\it locally finite},
i.e., there exists a constant $c>0$ such that
$$d_x:={\rm deg}(x) := \#\{ y \in V: (x,y) \in E\} \le c$$ for all $x\in G$. Let
$$
N(x)=\{ y \in V: (x,y) \in E, \ y\not= x\}.
$$
We may define $O(x):=N(x)\bigcup\{x\}$ an open neighborhood of $x$ and use them to define a topology on the graph $G$.
In fact, we can define the distance function $d(x_0,x_1)$ for $x_0,x_1\in V$. Introducing
the path set
$$
\mathbb{P}=\{P=x_0z_1...z_mx_1; z_i\in S\},
$$
and define the length of the path $P$ by
$$
L(P)=m+1
$$
Then we define
$$
d(x_0,x_1)=\inf_{P\in \mathbb{P}} L(P).
$$
Then $(V,d(\cdot,\cdot))$ is a complete metric space (in fact a length space in the sense of M.Gromov and then one may introduce the Hausdorff metric to handle the convergence of a sequence of graphs).

We give a remark here. In this metric space, the Monge problem always has a solution. Recall that for any two disjoint sets
$A=\{x_1,...,x_N\}$ and $B=\{y_1,...,y_N\}$, $A\bigcap B=\empty$, the Monge problem is to find a minimizer of the minimization problem
$$
M=\inf \{\sum_1^N d(x_i, y_{\sigma(i)}); \sigma:\{1,...,N\}\to \{1,...,N\} \ is \ a \ permutation\}.
$$
This is a minimization problem on finite set and it has at least one solution.

We say that $G$ is locally connected if each $N(x)$ is connected, i,e., for any $y_0,y_1\in N(x)$, there are two path $P_+=P_+(y_0,y_1)$ and $P_-=P_-(y_0,y_1)$ in $N(x)$  connecting
$x_0$ and $x_1$, i.e.,
$$P_+=y_0\sim x_1\sim ...\sim x_m\sim y_1, \ x_i\in N(x),$$
$$
P_-=y_0\sim z_1\sim ...\sim z_n\sim y_1, \ z_i\in N(x).
$$ Then we have the decomposition $N(x)=P_+\bigcup P_-$.
In this work we assume that $V$ is a infinite ($\sharp(V)=\infty$) locally finite and locally connected. Sometimes we write by $dx=d_x$.

Let $S$ be a finite subset of $V$, the {\it subgraph} $G(S)$
generated by $S$ is a graph, which consists of the vertex-set $S$
and all the edges $x \sim y, \ x,y \in S$ as the edge set.  The
boundary $\delta S$ of the induced subgraph $G(S)$ consists of all
vertices that are not in $S$ but adjacent to some vertex in $S$.We
assume that the subgraph $G(S)$ is connected. We now
recall some facts from the book \cite{Ch}. Sometimes people may like
to write
$$
\bar{S}=S\bigcup \delta S.
$$

\subsection{Gradient, divergence, and Laplacian}

For a function $f: S\bigcup \delta S\to {\mathbb C}$, let
$$
\nabla_{xy} f=f(y)-f(x)
$$
for $y\sim x$. Then
$$
\nabla_{xy} f^2=f(y)^2-f(x)^2=(f(y)+f(x))(f(y)-f(x))
$$
$$
=(\nabla_{xy} f)^2+2f(x)\nabla_{xy} f.
$$
Notice that there is a difference from the differential calculus.

For two function $f$ and $g$, we have
$$
\nabla_{xy} (fg)=f(y)g(y)-f(x)g(x)=f(y)\nabla_{xy} g+g(x)\nabla_{xy} f
$$
$$
=\nabla_{xy} f\nabla_{xy} g+f(x)\nabla_{xy} g+g(x)\nabla_{xy} f.
$$

We define the gradient of $f$ at $x\in S$ by
$$
\nabla f(x)=(\nabla_{xy} f)_{y\in N(x)}.
$$
A vector field $W$ on $S$ can be defined by
$$
W(x)=(w(xy))_{y\in N(x)}, \ w(xy)\in \mathbb{R}
$$
and the scalar product of two vector fields is given by
$$
W(x)\cdot U(x)=\frac{1}{d_x}\sum_{y\sim x} w(xy)u(xy).
$$
Hence
$$
|W(x)|^2=\frac{1}{d_x}\sum_{y\sim x} w(xy)^2.
$$
Of course, one may the concept of a vector filed on oriented edges.

\begin{Def} 1. For a function $f(x)$ and a vector field $W(x)$, we define the product between them by
$$
fW(x)=(\frac12 (f(x)+f(y))W(xy))_{y\in N(x)}.
$$

2. The directional derivative of $f$ along the vector field $W$ at $x\in S$ is defined by
$$
W(f)(x):=\frac{1}{d_x}W\cdot \nabla f(x)=\frac{1}{d_x}\sum_{y\in N(x)}w(xy)\nabla_{xy} f.
$$
\end{Def}

Then
$$
W(fg)(x)=\sum_{y\in N(x)}w(xy)\nabla_{xy} (fg).
$$
As a consequence of this, we have
$$
\nabla f(f)(x)=\sum_{y\in N(x)}\nabla_{xy} f\nabla_{xy} f=|\nabla f(x)|^2d_x.
$$
This means that our definition of the gradient is reasonable.

\begin{Def} For the vector field $W(x)$, we may define the divergence of it by
$$
div W(x)=\frac{1}{d_x}\sum_{y\in N(x)} w(xy).
$$
Then we have
$$
div \nabla f(x)=\frac{1}{d_x}\sum_{y\in N(x)} \nabla_{xy} f:=\Delta f(x).
$$
which is the Laplacian operator of the function $f$.
\end{Def}
So,
$$
 div (fW(x))=\frac{1}{d_x}\sum_{y\in N(x)} \frac12 (f(x)+f(y)w(xy)=f(x) div W(x)+\frac12 \nabla f(x)\cdot W(x).
$$

Note that for $w(xy)=-w(yx)$ for $x\sim y\in S$, we have
$$
\int_Sdiv W(x)=\sum_{x\in S} \sum_{y\in N(x)} w(xy)=0.
$$
This is a very useful formula and we shall this the \emph{Divergence theorem}. The proof is below. Set $w(xy)=0$ for any $x\nsim y$. Then
\begin{align*}
\sum_{x\in S} \sum_{y\in N(x)} w(xy)&=\sum_{x\in S} \sum_{y\in S} w(xy)\\
&=\frac12 \sum_{x\in S} \sum_{y\in S} w(xy)+\frac12 \sum_{x\in S} \sum_{y\in S} w(yx)=0.
\end{align*}

In general we may define a Hessian sum of two vectors $w_1=(w_1(xy))$ and $w_2=(w_2(xy))$ at $x$ by
$$
(w_1+_Hw_2)(x)=w_1(x)+_Hw_2(x)=(\frac12 (w_1(xy)+w_2(xz))_{y,z\in N(x)},
$$
which is a $d_x\times d_x$ matrix. It is trace is
$$
trace (w_1+_Hw_2)(x)=\sum_{y\in N(x)}\frac12 (w_1(xy)+w_2(xy)).
$$

We remark that we may define the Hessian matrix of $f$ at $x$ by
$$
H(f)(x)=\frac12 (f(y)+f(z)-2f(x))_{y,z\in N(x)},
$$
which is again a $d_x\times d_x$ matrix and it is $\nabla f(x)+_H\nabla f(x)$.
Note that $trace H(f)(x)=\Delta f(x)$.

We define the divergence of $H(f)$ by
$$
div H(f)(x)=d_x(f(y)-f(x))_{y\in N(x)}=d_x\nabla f(x).
$$

Its divergence of $w_1+_Hw_2$ at$x$ is defined to be
$$
div (w_1+_Hw_2)(x)=(w_1(xy)+w_2(xy))_{y\in N(x)}.
$$

With this understanding, we know the the maximum principle below.

\begin{Thm}
At any local minimum point of the function $f$ on $V$, we have $\nabla f(x)\geq 0$ is a non-negative vector and $H(f)(x)$ is a non-negative matrix,i.e., each element in the matrix is non-negative.
Then
$$
\Delta f(x)=trace H(f)(x)\geq 0.
$$
\end{Thm}

\subsection{Mini-max theorem}

We now try to define the mini-max value point for a real function on the locally finite and locally connected graph $G=(V,E)$.
For this purpose we need some preparation.

\begin{Def} 1. If there is a point $x\in S$ such that $\nabla_{xy} f\geq 0$ ($\nabla_{xy} f> 0$) for all
$y\in N(x):=\{y\in S; y\sim x\}$, we say $x$ is a local (strict) minimum point of $f$.

 2. If there is a point $x\in S$ such that $\nabla_{xy} f\leq 0$ ($\nabla_{xy} f< 0$) for all
$y\in N(x)$, we say $x$ is a local (strict) maximum point of $f$.

3. Given a point $x\in S$ such that $\sharp N(x)\geq 4$. If there is a path $x_0\sim x\sim x_1; \ x_0\in N(x), \ x_1\in N(x)$ such that
$\inf\{f(x_0),f(x_1)\}\geq f(x)$, and there are two points $y_0\in P_+(x_0,x_1)$ and $y_1\in P_-(x_0,x_1)$ such that $ f(y)\leq f(x)$ for $y_i$, we then call $x$ is a mini-max point of $f$.

4. We say $f$ is coercive if $f(x)\to \infty$ as $x\to \infty$.
\end{Def}

Then we have the following result.
\begin{Thm} Given a real function $f$ on a locally finite graph $G=(V,E)$ with $\sharp N(x)\geq 4$ for each $x\in V$. Assume that $f$ is coercive. Suppose that there are two local strict minimum points $z_0$ and $z_1$. Then there is a mini-max point of $f$.
\end{Thm}

The proof of this theorem is of a mountain pass type and we give it below.
\begin{proof}
Recall the path set
$$
\mathbb{P}=\{P=x_0z_1...z_mx_1; z_i\in S\},
$$
and we define
$$
c=\inf_{P\in \mathbb{P}} \sup_{z\in P} f(z).
$$
Note that $c>\inf\{f(x_0),f(x_1)\}$. Taking a minimizing path sequence $\{P_j\}$ such that
$$
\sup_{z\in P_j} f(z)\to c
$$
and
$$
\sup_{z\in P_j} f(z)\leq  c+1.
$$
By coercivity of $f$, we know that there are only finite many $P_j$'s with the property that $\sup_{z\in P_j} f(z)\leq  c+1$.
Then we know that there exists a path $P=x_0z_1...z_mx_1$ such that
$$
c=\sup_{z\in P} f(z)=f(z_0).
$$
Let
$$
Z=\{z\in P; f(z)=c\}=\{z_{i_j}\}_{j=1}^N.
$$
We start search from $z_{i_1}$. If there is a path
$$z_{i_1}\sim y_1\sim...\sim w; y_1\in N(z_{i_1})\backslash P, \ w_1\in P\bigcap N(z_{i_1})$$ such that
$$
f(y_m)< c, \forall y_m
$$
we may modify the path $P^x$ to get a new path $P_1$ by replacing $z_{i_1}$ by $y_1\sim...\sim w$.
If $Z=\{z_{i_1}\}$, then we have
$$
\sup_{z\in P^x} f(z)<c,
$$
which is impossible by the definition  of $c$.
So we are done in this case. In other cases, we search at $z_{i_2}$,
 and by induction we get a point $z\in Z$ such that
$$
f(y)\geq c,  \ \forall y\in N(z)\backslash P.
$$
and this $z$ is the mini-max we desired.

\end{proof}

\subsection{Volume and Cheeger constants}

Let $A,B\subset V$, we let
$E(A,B)$ is the set of edges with one endpoint in $A$ and one endpoint in $B$.

The volume of finite graph $S$ is
$$
|S|=vol(S)=\sum_{x\in S} d_x.
$$
We let $S^c$ be the complement of $S$ in $V$. Then $\delta S=\delta S^c$, which corresponds in  one-to-one way to the edge set $E(S,S^c)$.
 We denote by $|E(A,B)|$
the number of edges in $E(A,B)$. We define
$$
h_G(S)=\frac{|E(S,S^c)|}{\inf(vol(S),vol(S^c))}
$$
and the Cheeger constant (see also \cite{C})
$$
h_G=\inf_{S} h_G(S).
$$
The constant $h_G$ for $G=(V,E)$ can also be defined as
$$
h_G=\inf_{f\not=0} \sup_{c\in \mathbb{R}}\frac{\sum_{\{xy\in E\}} | \nabla_{xy} f|}{\int_V|f-c|}.
$$

We may also define
$$
g_G(S)=\frac{|\delta S|}{\inf(vol(S),vol(S^c))}
$$
and the Cheeger constant
$$
g_G=\inf_{S} g_G(S).
$$
Clearly $g_G\geq h_G$.

The open question is when the Cheeger constants can be achieved.

We now recall that the integration of $f$ over $S$ and the $L^2$ norm of $f$.
$$
\int_{\bar{S}} f = \sum_{x \in \bar{S}} f(x)d_x, \quad
\|f\|_S^2=\int_{\bar{S}}|f|^2=\sum_{x\in \bar{S}} f(x)\overline{f(x)}.
$$
The Dirichlet energy of $f$ on $S$ is defined by
$$
\|\nabla f\|_S^2=\sum_{\{x,y\in \bar{S}\}}|f(x)-f(y)|^2=\int_S|\nabla f|^2(x)dx.
$$

\subsection{Poincare inequality and Boundary conditions}

We say that $f:S\bigcup \delta S\to {\mathbb C}$ satisfies the {\it
Neumann boundary condition} if for all $x\in \delta S$,
$$
\sum_{\{y\in S;y\sim x\}} (f(y)-f(x))=0.
$$

Recall that the {\it Laplacian operator} can be written as
$$
(\Delta f)(x) = \frac{1}{d_x}\sum_{y: y \sim x} (f(y) - f(x)) = \frac{1}{d_x}\sum_{y: y \sim
x} \nabla_{xy} f.
$$
Also the Neumann condition can be written as
$$
\sum_{\{y\in S;y\sim x\}} \nabla_{xy} f=0, \quad \forall x \in
\delta S.
$$

We say that $f:S\bigcup \delta S\to {\mathbb C}$ satisfies the {\it
Dirichlet boundary condition} if  $f(x)=0$ for all $x\in \delta S$.

We now recall the Poincare inequality on the subgraph $S$ ( see Ch. 2 in \cite{Ch}).
Once we have $g_G>0$, we have the corresponding Poincare inequality
$$
\int_S |\nabla u|^2\geq c\int_S u^2, \ \ \forall u\in H_0^1(S)
$$
for some uniform constant $c>0$.

If $h_G>0$, we have the corresponding Poincare inequality
$$
\int_S |\nabla u|^2\geq c  \int_S (u-\bar{u})^2, \ \ \forall u\in X
$$
for some uniform constant $c>0$ where
$\bar{u}=\frac{1}{|S|}\int_S u$.

\subsection{Green formula}

The importance of the boundary conditions above is the boundary term
vanishing in the formula below.

\begin{Thm}\label{green} Assume that $f:S\bigcup \delta S\to {\mathbb C}$. Then we have
\begin{equation}\label{green+1}
\int_S (\Delta f,f)=-\frac{1}{2}\int_{S}|\nabla f|^2+\sum_{x\in
S}\sum_{y\in\delta S}\overline{f(x)}\nabla_{xy}f.
\end{equation}
\end{Thm}
This can be verified directly. In fact, we can directly verify the
following more general formula (see Theorem 2.1 in \cite{Ch} and \cite{MW1}) in a
compact form. When $f|_{\delta S}=0$, the last term is
$$
\sum_{x\in
S}\sum_{y\in\delta S}\overline{f(x)}\nabla_{xy}f=-\sum_{x\in
S}\overline{f(x)}f(x)=-\int_S |\nabla f|^2
$$
and then
$$
\int_S (\Delta f,f)=-\frac{3}{2}\int_{S}|\nabla f|^2.
$$
It is this formula used very often.

\begin{Thm}\label{green2} Assume that $f, \ g :S\bigcup \delta S\to {\mathbb R}$. Then we have
\begin{equation}\label{eq-theorem-3}
\int_{\bar{S}} (\Delta f,g)=-\frac{1}{2}\sum_{x,y\in \bar{S}
}(\nabla_{xy} f, \nabla_{xy} g)+\sum_{x \in S} \sum_{y\in \delta S: \ y \sim x} g(x)\nabla_{xy} f.
\end{equation}
\end{Thm}
In application we may assume that $f|_{\delta S}=0$. Then we have
$$
\int_{\bar{S}} (\Delta f,g)=-\frac{1}{2}\sum_{x,y\in \bar{S}
}(\nabla_{xy} f, \nabla_{xy} g)-\sum_{x \in S} g(x)f(x)d_x.
$$
 Then
\begin{equation}\label{green5}
\int_{\bar{S}} (\Delta f,g)-\int_{\bar{S}} (\Delta g,f)=-\sum_{x \in S} g(x)f(x)d_x-\sum_{x \in S} \sum_{y\in \delta S: \ y \sim x} f(x)\nabla_{xy} g.
\end{equation}

We remark that the formula (\ref{eq-theorem-3}) is proved in
\cite{Ch} for real functions, but the complex case can be done by
writing the complex function into the sum of real and imaginary
parts. When $S=V$ and then $\delta S=\emptyset$, the last term disappear and we need to assume that
$f, g\in H^1(V)$.

For convenience, we give the proof of Theorem \ref{green2} below.
\begin{proof} As remarked above, we need only prove the result for
real functions.

 We make the following computation
$$
\begin{array}{ll}
\int_S (\Delta f, g) & = \sum_{x \in S} (\Delta f (x), g(x)) = \sum_{x \in S} \sum_{y: \ y \sim x} (\nabla_{xy} f, g(x))\\
& = \sum_{x \in S} \sum_{y\in S: \ y \sim x} (\nabla_{xy} f, g(x)) + \sum_{x \in S} \sum_{y\in \delta S: \ y \sim x} (\nabla_{xy} f, g(x))\\
& = \sum_{x, y \in S: x\sim y}(\nabla_{xy} f, g(y) - \nabla_{xy} g) + \sum_{x \in S} \sum_{y\in \delta S: \ y \sim x} (\nabla_{xy} f, g(x)) \\
& = - \sum_{x,y \in S, x \sim y} (\nabla_{xy} f, \nabla_{xy} g) + \sum_{x,y \in S, x \sim y} (\nabla_{xy} f, g(y)) \\
& + \sum_{x \in S} \sum_{y\in \delta S: \ y \sim x} (\nabla_{xy} f,
g(x)) =: A + B + C
\end{array}
$$
where
$$
A=- \sum_{x,y \in S, x \sim y} (\nabla_{xy} f, \nabla_{xy} g),
$$
$$
B= \sum_{x,y \in S, x \sim y} (\nabla_{xy} f, g(y)),
$$
and
$$
 C=\sum_{x \in S} \sum_{y\in \delta S: \ y \sim x} (\nabla_{xy} f,
g(x))
$$
For the term $B$, we have
\[
\begin{array}{ll}
B & = \sum_{x,y \in S, x \sim y} (\nabla_{xy} f, g(y)) \\
& = \sum_{y \in S} (\sum_{x\in S \cup \delta S, x \sim y} \nabla_{xy} f, g(y)) - \sum_{y \in S} (\sum_{x\in \delta S, x \sim y} \nabla_{xy} f, g(y)) \\
& = \sum_{y\in S} (-\Delta f(y), g(y)) + \sum_{y \in S} (\sum_{x\in \delta S, x \sim y} \nabla_{yx} f, g(y)) \\
& = - \int_S (\Delta f, g) + C. \\
\end{array}
\]
(For the second term of the left hand side of the above equation, we
change $x$ to $y$,  $y$ to $x$, and see that this term is just $C$).
It follows that
\[
\int_S (\Delta f, g) = A + (- \int_S (\Delta f, g) + C) + C.
\]
Then we have
$$
2\int_S (\Delta f, g) = A  +2 C.
$$
We then re-write this into (\ref{eq-theorem-3}). The proof is
complete.
\end{proof}

We give two remarks here. One is below.
For $f:S\bigcup \delta S\to {\mathbb C}$ satisfying the Dirichlet
condition, the boundary term in Theorem \ref{green} can be written
as
$$
-\sum_{x\in S}\sum_{y\in\delta
S}\overline{(f(y)-f(x))}\nabla_{xy}f=-\sum_{x\in S}\sum_{y\in\delta
S}|\nabla_{xy}f|^2,
$$
which is real. Then
$$
\int_S (\Delta f,f)=-\frac{3}{2}\int_{S}|\nabla f|^2.
$$
This fact is useful in the estimation of $L^2$ and Dirichlet energy along the heat flow.

The other is that we may generalize the above formula in the following way.

\begin{Thm}\label{green5} Given a vector field $W$ on $\bar{S}$. Assume that $f :S\bigcup \delta S\to {\mathbb R}$. Then we have
\begin{equation}\label{eq-theorem-5}
\int_{\bar{S}} (div W,f)=-\frac{1}{2}\int_S
W\cdot \nabla f(x)+\int_{S} div(fW).
\end{equation}
\end{Thm}
Note that
$$
\int_S
W\cdot \nabla f(x)=\sum_{x\in S}\sum_{y\sim x}w(xy)(f(y)-f(x))
$$
and
$$
\int_S
div W(x) f(x)=\sum_{x\in S} div W(x)f(x)d_x=\sum_{x\in S}\sum_{y\sim x}w(xy)f(x)
$$
$$
=\frac12\sum_{x\in S}\sum_{y\sim x}w(xy)(f(x)-f(y))+\frac12\sum_{x\in S}\sum_{y\sim x}w(xy)(f(x)+f(y))
$$
$$
=-\frac12\int_S
W\cdot \nabla f(x)+\int_S div(fW).
$$

As we have mentioned above, the divergence term is zero provided there holds the asymmetry condition that $fW(xy)=-fW(yx)$ for $x\sim y$. This fact will be used in the discussion of harmonic maps from a graph to a Riemannian manifold.

\subsection{Eigenvalues and topology}

Given a function $Q:\bar{S}\to R$.
With Neumann or Dirichlet boundary condition, the Laplacian operator
on $S$ has finite eigenvalues  $\lambda_j\in \mathbb{R}$ (in monotone non-decreasing manner) with the corresponding
eigen-functions $\phi_j(x)$ \cite{Ch}, i.e.,
$$
(-\Delta+Q(x)) \phi_j=\lambda_j\phi_j, \ \ \ in \ \ S, \ \ \int_{\bar{S}}
|\phi_j|^2=1.
$$
In short, we can write this as
$$
L:=-\Delta +Q(x)= \sum_j \lambda_jI_j
$$
where $I_j$ is the projection on to the j-th eigenfunction $\phi_j$
of the induced subgraph $S$ (see p.145 in \cite{Ch}). We may define the Rayleigh quotient for nontrivial function $f$ by
$$
E(f)=\frac{\int_S fL fdx}{\int_S f^2dx}.
$$
Since $E(\lambda f)=E(f)$ for $\lambda\not=0$, we know that
$E$ is well-defined on the projective space
$$
P(X)=X\backslash \{0\}/ \mathbb{R}^{x}, \ \ \mathbb{R}^{x}=\mathbb{R}\setminus \{0\},
$$
where $$X:=H^1(S)=\{f:S\to R; \int_S|\nabla f|^2dx+\int_S |f(x)|^2dx<\infty\}.$$

For the Dirichlet boundary condition, we define the Sobolev space
$$
H_0^1(S)=\{f:S\to R; f|_{\delta S}=0, \ \int_S|\nabla f|^2dx+\int_S |f(x)|^2dx<\infty\}.
$$
We shall replace $X$ by $H_0^1(S)$ when we consider the problems with Dirichlet boundary condition. In particular, the first eigenvalue $\mu_1$ of $L$ with Dirichlet boundary condition is defined by
$$
\mu_1(S)=\inf\{\int_S fL fdx; \ f|_{\delta S}=0, \ \int_S f^2dx=1\}.
$$
We set $\mu_0=0$.

To simply matter, we only consider the Neumann boundary value problem we use an idea from M.Gromov.
Since for any nontrivial $f\in X$, if for some $\lambda\in \mathbb{R}$,
$$
L f=\lambda f, \ \ in \ S.
$$
Then
$E(f)=\lambda$. Using the Euler-Lagrange equation we know that the eigen-values of $L=-\Delta+Q$ (where $Q:S\to R$ is nontrivial in the sense that $\int_S Q\not=0$) corresponding to the critical values of the Rayleigh quotient $E$ on $P(X)$.
For any $\lambda\geq 0$, let
$$
X_\lambda=\{f\in X, E(f)\leq \lambda\}, \ \ X_0=span\{1\}=R,
$$
Let $$
X_j=span\{\phi_i, i\leq j\}.
$$
Note that $\lambda_0=0$ is not eigenvalue and $dim X_j=j$. Let $\lambda<\lambda_j$. For any subspace $Y\subset X$ with dimension $j$, we have a nontrivial function $f$, which is $L^2$ orthogonal to $X_{j-1}$. Then we have
$E(f)\geq \lambda_j$. Hence, $Y\subsetneq X_\lambda$, Then we have proven

\begin{Thm}\label{gromov}
The eigenvalues $\lambda_j$ is the minimal number of $\lambda$'s such that the level set
$$
P_\lambda=E^{-1}[0,\lambda]=\{f\in P(X), E(f)\leq \lambda\}\subset P(X)
$$
contains a projective subspace of dimension $j-1$.
\end{Thm}
There is a topological characteristic about $$P_{\lambda_j}=\{l\in P(X); E(f)\leq \lambda_j, \ \forall f\in l\},$$
 which we omit here. See \cite{G2}.

 One may define the principal eigenvalue of the the operator $L$ by
 $$
 \mu_1(Q)=\inf\{\mu_1(S); \ \forall S\subset V\}
 $$
 There is a Barta type result for estimation of the principal eigenvalue of the operator $L$. Namely,

 \begin{Thm}\label{barta}
 Assume that there is a positive function $u$ on $V$ such that
 $$
 Lu(x)\geq \mu u(x), \ \ for  \ \ x\in V.
 $$
 Then
 $$
 \mu_1(Q)\geq \mu.
 $$
 \end{Thm}
 \begin{proof} Assume $S\subset V$ be a finite subgraph of $V$ and $f$ is the first eigenfunction corresponding to $\mu_1(S)$.
 Then $f>0$ in $S$ and
 $$
 Lf=\mu_1(S) f, \ \ in \ S.
 $$
  We use the Green formula \eqref{green5} above for $f$ and $g=u$. Note that
 $$
 -\sum_{x \in S} g(x)f(x)d_x-\sum_{x \in S} \sum_{y\in \delta S: \ y \sim x} f(x)\nabla_{xy} g=-\sum_{x \in S} \sum_{y\in \delta S: \ y \sim x} f(x)g(y)\leq 0.
 $$
 and
 $$
 \int_{\bar{S}} (\Delta f,g)-\int_{\bar{S}} (\Delta g,f)=-\int_{\bar{S}} (Lf,g)+\int_{\bar{S}} (Lg,f)\geq (\mu-\mu_1(S))\int_{\bar{S}} (f,g).
 $$
 Then we have
 $$
(\mu-\mu_1(S))\int_{\bar{S}} (f,g)\leq -\sum_{x \in S} \sum_{y\in \delta S: \ y \sim x} f(x)g(y)\leq 0.
 $$
Then for any finite $S$,
$\mu\leq \mu_1(S)$.
Hence, $\mu\leq \mu_1(Q)$.
 This completes the proof.
 \end{proof}

\section{Part 2: Heat equations}\label{sect6}

We may study the transport equation given by
$$
f_t(t,x)=W(t,x)\cdot \nabla f(t, x), \ t>0, \  \ f(0,x)=f_0(x)
$$
for the unknown real function $f=f(t,x)$. Since this is an ODE for any time-variable vector field $W(t,x)$ on $[0,\infty)\times S$, we have a global existence of the solution by the fundamental theorem of ODE.
Note that
$$
\frac{d}{dt}\int_S f=\int_S f_t=\int_S W(t,x)\cdot \nabla f.
$$
In below, we shall introduce the Laplacian operator $\Delta$ so that we may study the diffusion equation
$$
u_t=\Delta u+W(t,x)\cdot \nabla u +V(t,x)u, \ \ (t,x)\in (0,\infty)\times S.
$$

To make notation more concise, we replace $\Delta$ defined above by $\frac23\Delta$, but still denoted by $\Delta$. The change will be very helpful in using the Green formula above.

\subsection{Heat kernels and Green functions}

We now turn attention to heat kernel. We define the heat kernel of $L$ as
$$
\mathbf{S}_t(x,y)=\sum_jexp(-\lambda_jt)\phi_j(x)\phi_j(y), \ \
t\geq 0.
$$
The heat kernel of $L$ can be written this as
$$
\mathbf{S}_t=\sum_j e^{-\lambda_jt}I_j.
$$
Then
$$
\mathbf{S}_t=e^{-tL }=I-tL+...., \ \ S_0=I.
$$
Formally we may define the Green function for $L$ (\cite{M1}) by
$$
G(\cdot,\cdot)=\int_0^\infty \mathbf{S}_t(\cdot,\cdot) dt.
$$

 For a function $f:S\bigcup \delta S\to {\mathbb C}$, we
define
\begin{equation}\label{sch}
u(x,t)=\sum_y\mathbf{S}_t(x,y)f(y), \ \  x\in S, \ t>0.
\end{equation}

Note that
$$
\mathbf{S}_0(x,y)=I=\sum_j\phi_j(x)\phi_j(y).
$$
Then for any $f:S\bigcup \delta S\to {\mathbb C}$,
$$
u(x,0)=\mathbf{S}_0(x,y)f(y)=\sum_y\sum_j\phi_j(x)\phi_j(y)f(y)=f(x).
$$
Then we can directly verify that the function $u$ satisfies the
heat equation
$$
\partial_tu(x,t)+L u(x,t)=0, \ \ x\in S, \ t>0,
$$
with $u(x,0)=f(x)$ for $x\in S$. We denote $V(S)$ the space of
functions $f:S\bigcup \delta S\to {\mathbb C}$ satisfy the Dirichlet
boundary condition and with the $L^2$ norm bound.

\subsection{Heat flow}

We show the following result.
\begin{Thm}\label{main}
Given a function $Q:S\bigcup \delta S\to {\mathbb R}$.
Assume that the function $f:S\bigcup \delta S\to {\mathbb R}$ satisfies the Dirichlet
boundary condition. Then there is a global solution $u(t):S\bigcup
\delta S\to {\mathbb R}$, which can be expressed in (\ref{sch}),
such that $u$ satisfies the heat equation
$$
\partial_tu(x,t)+Lu(x,t)=0, \ \ x\in S, \ t>0,
$$
with $u(x,0)=f(x)$ for $x\in S$ and with the L2 conservation
$$
\|u(t)\|^2+\int_0^t(2|\nabla u|^2+2(Vu,u))= \|f\|^2.
$$
and the energy identity
$$
F(u(t))+2\int_0^t dt\int_S |u_t|^2=F(f).
$$
where
$$
F(u)=\int_{\bar{S}}|\nabla u(t)|^2+(Vu,u).
$$
\end{Thm}
The proof will be given below by the use of the spectrum
of the operator $L$ and the Green formula.

 We now use Theorem \ref{green} (and Theorem
\ref{green2}) to prove Theorem \ref{main}.

\begin{proof}
Assume that $f$ satisfies either  Neumann or Dirichlet boundary
condition. Compute, via the use of the formula (\ref{green+1}),
$$
\frac{d}{dt}\|u(t)\|^2=-2(L u,u)=-2|\nabla u|^2+2(Vu,u).
$$
where we have used implicitly the boundary condition which implies
that $u_t=0$ on $\delta S$.
 We then have the  conservation
$$
\|u(t)\|^2+\int_0^t(2|\nabla u|^2+2(Qu,u))= \|f\|^2.
$$
Recall $$
F(u)=\int_{\bar{S}}|\nabla u(t)|^2+(Qu,u)
$$
Similarly we compute
$$ \frac{d}{dt}F(u)=2 (\nabla u,\nabla u_t)+2(Vu,u_t).
$$
Using the formula (\ref{eq-theorem-3}) for $f=u$ and $g=u_t$, we
know that
$$
2(\nabla u,\nabla u_t)=-2\int_{\bar{S}} (\Delta u,u_t).
$$
Using the heat equation above, we know that the term in the right
side is
$$ -2\int_{S} F(u)=-2\int_S |u_t|^2.
$$
 Then
$$
F(u(t))+2\int_0^t dt\int_S |u_t|^2=F(f).
$$

The uniqueness of the solution follows from the above inequalities.
This completes the proof.
\end{proof}

We
remark that the above approach can not be used to handle the case when $Q(x)$ is replaced by $V(t,x)$; i.e., the time variable
potential heat equation
$$
u_t-\Delta u+V(t,x)u=0.
$$
This is a goal of next section, where we introduce the discrete Morse flow to treat this case.

\subsection{Discrete Morse flow with prescribed time-variable quantity}

The result above motivates us to study the following problem.
For any given $t-$ continuous function $V(t)$. We now consider discrete Morse flow for the heat flow
\begin{equation}\label{ht2}
\partial_t u=\Delta u+V(t) u,
\quad \text{on $[0,T]\times S$}
\end{equation}
with the parabolic initial-boundary value $u_0=\phi$, which is in $H_0^1(S)$,
where $S\subset V$ is a finite sub-graph. Here again, we fix $T>0$ as before and let $h=T/N$ for some positive integer $N>1$.
For $n=1,...,N$, we let $t_n=nh$ and $\lambda_n=V(t_{n-1})$. Define $\lambda_N(t) =\lambda_n$ for $t\in (t_{n-1},t_n]$. Note that $|V(t)-\lambda_N(t)|\to 0$ on $[0,T]$ as $h\to 0$.

One may define the discrete Morse for \eqref{ht2} by the following way.

To introduce the  discrete Morse flow for  flow \eqref{ht2} , we define, for the data $u_{n-1}\subset H:=H_0^1(S)$, the functional
$$
J_n(u)=\frac12 \int_S |\nabla u|^2dx -\frac12 \lambda_n\int_S |u|^2dx,
$$
and
the functional for the $n$-step,
$$
F_n(u)=\frac{1}{2h}\int_S |u-{u}_{n-1}|^2dx+J_n(u).
$$
Note that by Poincare inequality we have
$J_n(u)\geq 0$ for any $u\in H$. Since $\lambda_n\geq \lambda_{n-1}$ we also have
$$
J_n(u)\leq  J_{n-1}(u).
$$

Then we know that there is a minimizer $u_n\in H$ of this functional $F_n(u)$, which solves the following equation
\begin{equation}\label{EL5}
\frac{1}{h}\left({u}_n-{u}_{n-1}\right)
=\Delta {u}_n+\lambda_n u_n \quad
\text{on $S$.}
\end{equation}
in the weak sense.

By the minimality of $u_n$, we have
$$
\frac{1}{2h}\int_S |u-{u}_{n-1}|^2dx+J_n(u_n)\leq J_n(u_{n-1})\leq J_{n-1}(u_{n-1}).
$$
Adding them together, we have
\begin{equation}\label{bounds5}
\frac{1}{2}\int_0^T\int_S |\frac{u_n-{u}_{n-1}}{h}|^2dx+J_N(u_N) \leq J_0(u_0).
\end{equation}

We now define  ${u}_N(t)\in H^1_0$ and $\hat{u}_N(t)\in H$ for $t\in [0,T]$ in such a way
that, for $n=1,...,N$,$t_n=nh$,
\[
u_N(t)=\frac{t_n-t}{h}u_{n-1}+\frac{t-t_{n-1}}{h} u_n; \ \  t\in [t_{n-1},t_n].
\]
Define
$$
\hat{u}_N(t)={u}_n \quad t\in (t_{n-1},t_n]
$$
and
$$
\hat{u}_N(t)=\phi,  \quad t\in [-h,0].
$$

We further define, for $n=1,...,N$,
\[
\partial_t {\hat{u}_N(t)}=\frac{1}{h}\left(u_n-u_{n-1}\right),
\quad t\in [t_{n-1},t_n].
\]

Note that
$$
\partial_t {\hat{u}_N(t)}=\partial_t {u}_N(t).
$$
Then,
$$
u_N(t)=\phi+\int_0^t\partial_t {\hat{u}_N(t)}
$$
and
$$
|u_N(t)|_{L^2}=|\phi|_{L^2}+\int_0^t|\partial_t {\hat{u}_N(t)}|_{L^2}.
$$
Hence, we have
$$
|u_N(t)|^2_{L^2}\leq2|\phi|_{L^2}^2+|\int_0^t|\partial_t {\hat{u}_N(t)}|_{L^2}|^2.
$$

By (\ref{bounds5}) we have for some uniform constant $C(T)>0$,
\[
\int_0^T |\partial_t {u}_N(t)|^2_{L^2}\leq C(T)
\]
and
$$
\int_0^T |\partial_t {u}_N(t)|_{L^2}\leq C(T).
$$
The latter implies that
 for any $t>0$,
 $$
 |u_N(t)-u_0|_{L^2}\leq C(T),
 $$
i.e., for each $t>0$,
$$
|u_N(t)|_{L^2}\leq C(T)+|u_0|_{L^2}:=C_1.
$$
Since $J_N(u_N(t))\leq J_0(u_0)$, we know that
$$
\frac12 \int_S |\nabla u_N(t)|^2dx\leq J_0(u_0)+|V|_\infty C_1^2.
$$
Then as before, we may get the limit $u_\infty(t)$ which solves \eqref{ht2} weakly for the $t-$ continuous function $V(t)$.

We now have the following conclusion.

\begin{Thm}\label{mali4} Fix any continuous bounded function $V(t)$ for $t\in [0,\infty)$.
For any $T>0$ and any $\phi\in H^1$ with $u_0=\phi $, there exists a discrete Morse flow $\{\hat{u}_N(t)\}$ as above and its limit is a weak
solution to \eqref{ht2} on $[0,T]\times S$ in the sense of the weak
form \eqref{ht2}.
\end{Thm}

Using a similar argument we may have
\begin{Thm}\label{mali4} Fix any $t-$ continuous function $V(t,x)$ for $(t,x)\in [0,\infty)\times \bar{S}$.
For any $T>0$ and any $\phi\in H^1_0$ with $u_0=\phi $, there exists a discrete Morse flow $\{\hat{u}_N(t)\}$ as above and its limit is a weak
solution to \eqref{ht3}
\begin{equation}\label{ht3}
\partial_t u=\Delta u+V(t,x) u,
\quad \text{on $[0,T]\times S$}
\end{equation}
with the parabolic initial-boundary value $u_0=\phi$, which is in $H_0^1(S)$,
in the sense of the weak
form \eqref{ht3}.
\end{Thm}

\section{Part 3. Discussions about harmonic maps}
  In this section, we discuss the definitions of harmonic maps from a graph to a Riemannian manifold and harmonic morphism from a graph to another graph.
  We also formally formulate the harmonic map equation for a mapping from a graph to the Riemannian manifold by computing the Euler-Lagrange equation of the energy functional. We start from the definition of a energy density of a mapping from a graph to the Riemannian manifold.

  Assume $(M,g)$ is a complete (or compact) Riemannian manifold and we consider the harmonic maps from the graph $V$ to $M$. Denote by $d_M$ the distance function on $M$.
  For any map from $V$ to $M$, we define the energy density by
  $$
  e(f)(x)=\frac12 \frac{1}{d_x}\sum_{y\in N(x)} d_M^2(f(x),f(y))
  $$
and its energy on the subset $S$ of $V$  by
$$
E_S(f)=\int_{\bar{S}} e(f)(x)dx.
$$
The critical point of the functional $E_S(\cdot)$ is called the harmonic map from $S$ to $M$. To formally compute the first variation of $E_S$, we may isometrically embed $M$ into an Euclidean space $R^K$ and let $O(M)$ be an open neighborhood of $M$ in $R^k$. Denote by $P:O(M)\subset R^K\to M$ the projection. Suppose that $u$ is a harmonic map from $V$ to $M$. For any
$\eta:V\to M$, $\eta|_{\delta S}=0$, and for small $t$, we consider the variation $P(u+t\eta)$ of $u$. Then at $t=0$, using the Divergence theorem,
\begin{align*}
&\frac{d}{dt} E_S(P(u+t\eta))=\frac12\frac{d}{dt} \sum_x  \sum_{y\in N(x)}d_M^2(P(u(x)+t\eta(x)),P(u(y)+t\eta(y)) \\
 &=\sum_x \sum_{y\in N(x)}(d_M(u(x),u(y))(\partial_1d_M (u(x),u(y))\eta(x)+\partial_2d_M (u(x),u(y))\eta(y))\\
 &= \sum_x\sum_{y\in N(x)}(d_M(u(x),u(y))(\partial_1d_M (u(x),u(y))+\partial_2d_M (u(x),u(y)))\eta(x)+\\
&
+\sum_x\sum_{y\in N(x)}d_M(u(x),u(y))\partial_2d_M (u(x),u(y))(\eta(y)-\eta(x))\\
&
= \int_SdxM(u)(x)\eta(x)+\int_S dx W(x)\cdot \nabla \eta(x)\\
&
= \int_Sdx M(u)(x)\eta(x)-2\int_Sdx div W(x)\eta(x),
\end{align*}
where
$$
M(u)(x)=\frac{1}{d_x}\sum_{y\in N(x)}(d_M(u(x),u(y))(\partial_1d_M (u(x),u(y))+\partial_2d_M (u(x),u(y)))
$$
and
$$
W(x)=(d_M(u(x),u(y))\partial_2d_M (u(x),u(y))_{y\in N(x)}.
$$
Then we have the harmonic map equation
$$
-div W(x)+\frac12M(u)(x)=0, \ in \ S.
$$
The corresponding heat flow equation is
$$
u_t(x)=div W(x)-\frac12M(u)(x), \ \ x\in S
$$
with suitable initial and boundary conditions. Even the target is a non-compact complete Riemannian manifold, we still expect that this flow has a global solution with the initial and boundary conditions.

Assume that $(M,g)$ is a compact Riemannian manifold.
With a given initial data with boundary condition on a finite subgraph $S$, one may use the heat flow method to study the corresponding heat flow equation for the energy functional $E_S$.
In particular, for any boundary data $\varphi:\delta S\to M$,
We define the mapping class
$$
\mathbb{A}_\varphi=\{u: S\bigcup \delta S\to M; u|_{\delta S}=\varphi\}.
$$
This mapping class is always non-empty when $S$ is finite.
We say $u$ is a harmonic map from $S$ to $M$ with Dirichlet boundary condition $\varphi$ if
$u$ is a critical point of $E_S$ in the mapping class $\mathbb{A}_\varphi$.
 Then in this case, by using the direct method, there exists at least a harmonic map, which is a minimizer of the energy functional in this mapping class. One expects when the Riemannian manifold $(N,g)$ has negative curvature, the minimizer is the unique harmonic map from $S$ to $M$ in the mapping class.

  Given two graphs $V_1$ and $V_2$ and we may define the harmonic morphism $F$ from $V_1$ to $V_2$ such that for any harmonic function $h$ on $V_2$, the composition $h\circ F: V_1\to R$ is a harmonic function. Generally speaking, it is hard to find the existence of harmonic morphism defined in this way. Of course, we may localize the definitions above and we leave them for interested readers.

\end{document}